\documentclass[10pt]{article}
\usepackage{geometry, enumerate, amsthm, amsmath, amssymb, amscd, color}
\usepackage{hyperref}
\usepackage{courier}
\usepackage[all]{xy}

\title{Idempotent pairs and PRINC domains}

\author{Giulio Peruginelli, Luigi Salce and Paolo Zanardo\footnote{Department of Mathematics, University of Padova, Via Trieste 63,  35121 Padova, Italy. E-mail: gperugin@math.unipd.it, salce@math.unipd.it, pzanardo@math.unipd.it. Research supported by  ``Progetti di Eccellenza 2011/12" of Fondazione CARIPARO and by the grant "Assegni Senior" of the University of Padova.}} 
  
\numberwithin{equation}{section}
\newtheorem{Proposition}[equation]{Proposition}
\newtheorem{Theorem}[equation]{Theorem}
\newtheorem{Lemma}[equation]{Lemma}
\newtheorem{Corollary}[equation]{Corollary}
\theoremstyle{definition}
\newtheorem{definition}[equation]{Definition}

\newtheorem{Rem}[equation]{Remark}
\theoremstyle{definition}

\def\Z{{\mathbb Z}}
  \def\Q{{\mathbb Q}}
  \def\F{{\mathfrak f}}

\begin{document}

\maketitle

\begin{center}
\emph{{\small Dedicated to Franz Halter-Koch on the occasion of his 70th birthday}}
\end{center}
\vskip0.3cm
\begin{abstract}
\noindent A pair of elements $a,b$ in an integral domain $R$ is an idempotent pair if either $a(1-a) \in bR$, or $b(1-b) \in aR$. $R$ is said to be a PRINC domain if all the ideals generated by an idempotent pair are principal. We show that in an order $R$ of a Dedekind domain every regular prime ideal can be generated by an idempotent pair; moreover, if $R$ is PRINC, then its integral closure, which is a Dedekind domain, is PRINC, too. Hence, a Dedekind domain is PRINC if and only if it is a PID. Furthermore, we show that the only imaginary quadratic orders $\Z[\sqrt{-d}]$, $d > 0$ square-free, that are PRINC and not integrally closed, are for $d=3,7$.
\end{abstract}
\medskip
\noindent{\small \textbf{Keywords}: Orders, Conductor, Primary decomposition, Dedekind domains, Principal ideals, Projective-free.\\ \textbf{2010 Mathematics Subject Classification}: 13G05, 13F05, 13C10, 11R11.}


\section*{Introduction.}

Let $R$ be an integral domain, $M_n(R)$ the ring of matrices of order $n$ with entries in $R$, $T$ any singular matrix in $M_n(R)$ (i.e. $\det T = 0$). A natural question is to find conditions on $R$ to ensure that $T$ is always a product of idempotent matrices. This problem was much investigated in past years, see \cite{SZ} for comprehensive references. The case when $R$ is a B\'ezout domain (i.e. the finitely generated ideals of $R$ are all principal) is crucial. In fact, we recall three fundamental results, valid for matrices with entries in a B\'ezout domain. Laffey \cite{L} showed that every square singular matrix $T$ with entries in an Euclidean domain is a product of idempotent matrices if and only if the same property is satisfied just for $2 \times 2$ matrices. This result was extended  to B\'ezout domains in \cite{SZ}. Arguably, the most important result in this subject is due to Ruitenburg \cite{R}, who proved that every singular matrix $T \in M_n(R)$ is a product of idempotent matrices if and only if every invertible matrix $U \in M_n(R)$ is a product of elementary matrices. From Ruitenburg's theorem and a result by O'Meara \cite{O}, we derive that every singular matrix with entries in a B\'ezout domain $R$ is a product of idempotents if and only if $R$ admits a {\it weak Euclidean algorithm} (see \cite{SZ} for the definitions and other details). As a matter of fact, a major problem in this subject is to establish whether the property that any singular matrix $T$ with entries in $R$ is a product of idempotent matrices implies that $R$ is a B\'ezout domain. 

While investigating this problem in \cite{SZ}, the second and third authors were led to define the property (princ) of an integral domain $R$. We rephrase the property in a way more suitable to our discussion.  

Two elements $a,b$ of a commutative integral domain $R$ are said to form an {\it idempotent pair} if they are the entries of a row, or of a column, of a $2 \times 2$ non-zero idempotent singular matrix. Since a non-zero matrix $T = (a_{ij}) \in M_2(R)$, different from the identity, is idempotent if and only if $\det(T) = 0$ and $T$ has trace $1$, it is easily seen that $a, b \in R$ form an idempotent pair if and only if either $a(1-a) \in bR$, or $b(1-b) \in aR$. We say that an integral domain $R$ satisfies the (princ) property if all the ideals generated by idempotent pairs are principal; such a ring will be called {\it PRINC} domain. The class of PRINC domains obviously includes B\'ezout domains. This class of domains was also investigated by McAdam and Swan in \cite{MS1} and \cite{MS2} under the name UCFD (unique comaximal factorization domain) and from a different point of view (see Remark \ref{McASw}).

In \cite{SZ} it is proved that if a PRINC domain satisfies the condition that $2 \times 2$ singular matrices are product of idempotent matrices, then it is necessarily a B\'ezout domain. A similar result was proved recently in \cite{B} by Bashkara Rao, who assumed the domain $R$ to be projective-free instead than PRINC; recall that  a domain $R$ is projective-free if finitely generated projective $R$-modules are free. Actually, Bashkara Rao's result  is a particular case of the mentioned result of \cite{SZ}, since the class of PRINC domains, as proved in \cite{SZ}, includes, besides B\'ezout domains, the projective-free domains and the unique factorization domains. In \cite{SZ} it was also claimed (without proof) that the ring $\Z[\sqrt{-3}]$ is an example of a non-integrally closed PRINC domain. 

In Section 1 of this paper we provide a characterization of ideals generated by idempotent pairs, related to invertible ideals in domains of finite character and PRINC domains.

In Section 2 we consider orders in rings of algebraic integers. We prove that a prime ideal of such an order $O$ which is comaximal with the conductor of $O$ is generated by an idempotent pair. This fact implies that the maximal ideals of Dedekind domains are generated by idempotent pairs; it follows that a Dedekind domain which is not a PID cannot be a PRINC domain. Another relevant consequence is that if the order $O$ is a PRINC domain, then its integral closure is necessarily a PID. However, concerning the orders $\Z[ \sqrt{-d}]$ in quadratic imaginary extensions of $\Q$ whose integral closures are PIDs, we prove that they are not PRINC domains, when $d=11,19, 43, 67, 163$. 

On the other hand, in Section 3 we prove that the rings $\Z[ \sqrt -3]$ and $\Z[ \sqrt -7]$ are PRINC domains.Therefore, from this point of view, $\Z[ \sqrt -3]$ and $\Z[ \sqrt -7]$ are exceptional among the imaginary quadratic orders $\Z[\sqrt{-d}]$, $d > 0$ square-free. We also prove that the invertible ideals of these two rings are principal, and from this fact we deduce that they are projective-free. Let us note that a first proof that $\Z[ \sqrt -3]$ is a PRINC domain was privately communicated by U. Zannier to the third author; that proof used arguments different from those used in Theorem \ref{primeidealsZsqrt-3-7} of the present paper.

\vskip0.4cm

\section{Ideals generated by an idempotent pair}

In what follows, every ring $R$ considered will be a commutative integral domain. Some results proved in this section are valid also for commutative rings. If $I$ is an ideal of $R$, generated by $x_1, \dots, x_n \in R$, we will use the notation $I = \langle x_1, \dots, x_n \rangle$.

We recall the definitions given in the introduction: $a, b \in R$ are said to be an idempotent pair if either $a(1-a) \in bR$, or $b(1-b) \in aR$, and $R$ satisfies property (princ) if all the ideals generated by idempotent pairs are principal. For short, we will say that $R$ is a PRINC domain.

\begin{Lemma} \label{UFD}
Let $I$ be an invertible ideal $I$ of a UFD $R$. Then $I$ is principal.
\end{Lemma}

\begin{proof}
Assume, for a contradiction, that the minimal number of generators of $I$ is $n \ge 2$, say $I = \langle x_1, \dots, x_n \rangle$ (in particular, $I$ is a proper ideal). We may assume, without loss of generality, that $GCD(x_1, \dots, x_n) = 1$. Take any element $y/z \in I^{-1}$, where $GCD(y,z) = 1$. Then, from $x_i y/z \in R$ for all $i \le n$, it follows that $z $ divides $x_i$ for every $i \le n$. We conclude that $z$ is a unit of $R$, hence $I^{-1} \subseteq R$, since $y/z$ was arbitrary. Then $R = I I^{-1} \subseteq I$, impossible. \qed
\end{proof}

Recall that two ideals $I$ and $J$ of a commutative ring $R$ are said to be comaximal if $I+J=R$.

\begin{Lemma}\label{referee}
Let $I$ and $J$ be two comaximal ideals of a commutative ring $R$. Then there exists an element $a\in I$ such that $a-1\in J$, implying that $a^2-a\in IJ$ and $I=aR+IJ$.
\end{Lemma}
\begin{proof}
Since $I+J=R$, there exists an element $a\in I$ such that $a-1\in J$. Hence $a^2-a=a(a-1)\in IJ$. Also, $aR+IJ=aR+I(aR+J)=aR+I=I$, since $aR$ and $J$ are comaximal. Thus $I=aR+IJ$. Similarly, $J=(a-1)R+IJ$. \qed
\end{proof}

The next theorem gives different characterizations for ideals generated by an idempotent pair.

\begin{Theorem}\label{idempotents}
Let $I$ be a non-zero ideal of an integral domain $R$. The following conditions are equivalent:
\begin{enumerate}
\item $I$ is generated by an idempotent pair.
\item $I=\langle a,b\rangle=\langle a^2,b\rangle$, for some $a,b\in R$.
\item There exists an ideal $J$ such that $I$ and $J$ are comaximal and such that $IJ$ is a principal ideal.
\end{enumerate}
In particular, if one of the above equivalent conditions holds, then $I$ is invertible.
\end{Theorem}
\begin{proof} $(1)\Rightarrow (3)$: Suppose that $I=\langle a, b \rangle$, where $a,b$ is an idempotent pair, so $a^2-a=bc$, for some $c\in R$. Let $J=\langle a-1, b \rangle$. Then $I$ and $J$ are comaximal and we have
$$
IJ=\langle a, b \rangle \langle a-1, b \rangle =\langle bc, ab, b(a-1), b^2 \rangle = b \langle c, a, a-1,  b \rangle = b R,
$$

$(3)\Rightarrow (2)$: By assumption there exists an ideal $J$ such that $I+J=R$ and $IJ=bR$, for some $b\in R$. By Lemma \ref{referee}, there exists $a\in I$ such that $I=\langle a, b \rangle$, with $a^2-a\in bR$. In particular, it follows that $\langle a, b \rangle=\langle a^2, b \rangle$.

$(2)\Rightarrow (1)$: Since $\langle a, b \rangle=\langle a^2, b \rangle$, there exist $\lambda,\mu \in R$ such that $a=\lambda a^2+\mu b$. This implies that $a-\lambda a^2 =a(1-\lambda a)=\mu b$, so $\lambda a(1-\lambda a)=\lambda\mu b$, and $\lambda a$, $ b$ form an idempotent pair. Obviously, $\langle \lambda a , b \rangle \subseteq I$. Conversely, since $a\in \langle \lambda a, b \rangle $, also $a^2\in  \langle \lambda a, b \rangle$, consequently $I= \langle  a^2 , b \rangle \subseteq \langle \lambda a, b \rangle$. So $I= \langle \lambda a, b \rangle$.

The last claim follows immediately from the characterization at the point (3).  \qed
\end{proof}

Recall that $R$ is said to be a B\'ezout domain if every finitely generated ideal of $R$ is principal, and $R$ is called projective-free if every finitely generated projective $R$-module is free. 
\begin{Proposition} \label{Bezout}
If an integral domain $R$ is either B\'ezout, or UFD, or projective-free, then $R$ satisfies property (princ).
\end{Proposition}

\begin{proof}
A B\'ezout domain is trivially a PRINC domain. Moreover, by Theorem \ref{idempotents}, every ideal $I$ of $R$ generated by an idempotent pair is invertible. Then such $I$ is free, hence principal, when $R$ is projective-free. Finally, if $R$ is a UFD, then $I$ is principal by Lemma \ref{UFD}. \qed
\end{proof}

\begin{Corollary} \label{generalcase}
Let $R$ be an integral domain, $b$ an element of $R$ such that $bR$ is a finite product of primary ideals that are pairwise comaximal. Let $Q$ be such a primary ideal and let $P$ denote its radical. Then there exists $a \in Q$ such that $a, b$ form an idempotent pair, and $P$ is the radical of $\langle a, b \rangle$.
\end{Corollary}

\begin{proof} By assumption, $bR=QJ$, where $Q$ and $J$ are comaximal. Hence, by  Lemma \ref{referee} (or Theorem \ref{idempotents}, (3)), there exists $a\in Q$ such that $a,b$ is an idempotent pair and $Q=\langle a,b\rangle$. \qed
\end{proof}

In the case of a domain of finite character, the last claim of Theorem \ref{idempotents} can be reversed. We recall that a domain is said to be of finite character if each non-zero element is contained in finitely many maximal ideals; moreover, an invertible ideal $I$ of a domain of finite character is $1\frac{1}{2}$ generated, that is, $I$ is generated by two elements and one of the two generators can be arbitrarily chosen among the non-zero elements of the ideals (see \cite[Proposition 2.5, (e), p. 12]{FS}). 

\begin{Corollary} Let $R$ be an integral domain of finite character and $I$ an invertible ideal. Then $I$ is generated by an idempotent pair.
\end{Corollary}
\begin{proof}
Choose $0\not=a\in I$. By the aforementioned Proposition 2.5 in \cite{FS}, there exists $b\in I$ such that $\langle a^2, b \rangle =I$. Now, $I=\langle a^2, b \rangle \subseteq \langle a, b \rangle\subseteq I$, so each of the previous containments is indeed an equality. By Theorem \ref{idempotents}, $I$ is generated by an idempotent pair. \qed
\end{proof}

As an application to domains of finite character, we derive the following 

\begin{Corollary}\label{PruferfinitePrinc}
Let $R$ be a domain with finite character. Then $R$ is a PRINC domain if and only if all invertible ideals are principal. In particular, if $R$ is also Pr\"ufer, then $R$ is a B\'ezout domain.
\end{Corollary}

\begin{proof}
The sufficiency holds for any domain $R$, by \cite[Proposition 4.2]{SZ}. The necessity follows from the preceding proposition. \qed
\end{proof}

\begin{Rem}\label{McASw}
McAdam and Swan defined the notion of an $S$-ideal $I$ in \cite{MS1} as in condition 2. of Theorem \ref{idempotents}, in the context of a definition analogue to that of unique factorization domain. We recall this definition here. Let $R$ be an integral domain. A nonzero non-unit element $b$ of $R$ is pseudo-irreducible if it is not possible to factor $b$ as $b=cd$ with $c$ and $d$ comaximal non-units. The domain  $R$ is called {\it comaximal factorization domain} (CFD) if any nonzero non-unit element $b$ has a complete comaximal factorization, namely, $b=b_1\cdot\ldots\cdot b_m$, where the $b_i$'s are pairwise comaximal pseudo-irreducible elements of $R$. A CFD is a {\it unique comaximal factorization domain} (UCFD) if complete comaximal factorizations are unique. In \cite[Theorem 1.7]{MS1} they show that if $R$ is a  CFD, then $R$ is UCFD if and only if every $S$-ideal is principal, that is, $R$ is a PRINC domain, by  Theorem \ref{idempotents}. They prove also that a domain with finite character is a CFD (\cite[Lemma 1.1]{MS1}), from which our Corollary \ref{PruferfinitePrinc} also follows.
\end{Rem}

\section{Orders in number fields and idempotent pairs}\label{orders}

We recall the definition of order.

\begin{definition}
An integral domain $O$ is an order if its integral closure $D$ in its quotient field is a Dedekind domain which is finitely generated as an $O$-module.
\end{definition}

By a well-known result of Eakin (\cite{Ea}), it follows that $O$ is Noetherian as well. So, an order is a one-dimensional Noetherian domain. We say that an order is proper if it is not integrally closed. We recall that a Dedekind domain is characterized by the fact that each ideal can be written uniquely as an intersection, or equivalently as a product, of powers of prime ideals. On the other hand, since an order is a one-dimensional Noetherian domain, each ideal of an order can be written uniquely as a product of primary ideals (see for example \cite[Theorem 9, Chapt. IV, \S. 5, p. 213]{ZS1}). It follows that in a proper order there are some primary ideals of the order $O$ which are not equal to a power of a prime ideal. 
We recall that the conductor of the integral closure $D$ of an order $O$ is defined as
$$
\mathfrak{f}\doteqdot(O:D)=\{x\in O\mid xD\subseteq O\}.
$$
The conductor is the largest ideal of $O$ which is also an ideal of $D$. Since $D$ is a finitely generated $O$-module, $\mathfrak{f}$ is non-zero. Following the terminology of \cite{Neu}, we call an  ideal of $O$ (or of $D$) comaximal with $\mathfrak{f}$ a {\it regular} ideal.  

We will need the following easy fact.

\begin{Lemma} \label{Neu}
Let $O$ be an order and $P$ a prime ideal. Then the set of $P$-primary ideals of $O$ is linearly ordered if and only if $P$ is regular. If this condition holds, then each $P$-primary ideal is equal to a power of $P$. 
\end{Lemma}

\begin{proof}
Since the local ring $O_{P}$ is a local one-dimensional noetherian ring, every non-zero ideal is $PO_P$-primary. We recall that there is a one-to-one correspondence between $P$-primary ideals of $O$ and $PO_P$-primary ideals of $O_P$. Hence, the set of $P$-primary ideals of $O$ is linearly ordered if and only if the set of ideals of $O_P$ is linearly ordered, that is, $O_P$ is a valuation domain. By \cite[Proposition 12.10]{Neu} this condition is equivalent to the fact that $P$ is a regular prime. 

If one of these equivalent conditions holds, then $O_P$ is a DVR, so its ideals are powers of the maximal ideal $PO_P$. Hence, the $P$-primary ideals of $O$ are powers of $P$. \qed
\end{proof}

The following result is a consequence of Corollary \ref{generalcase}, since an order is a domain of finite character; we include a direct proof, which makes use of a different technique.

\begin{Proposition}\label{primeidealsorderPrinc}
Let $O$ be an order and $P$ a regular prime ideal of $O$. Then there exists an idempotent pair that generates $P$. In particular, if $O$ is integrally closed, then every prime ideal is generated by an idempotent pair.
\end{Proposition}

\begin{proof}
Let
$$
\mathfrak{f}=\prod_{i=1}^k Q_i
$$
be the primary decomposition of $\mathfrak{f}$ in $O$, where the $Q_i$'s  are primary ideals of $O$ with distinct radicals $P_i$, $i=1,\ldots,k$. Now take $x \in P \setminus P^2$. Since $P$ is comaximal with $\mathfrak{f}$, there exists $b\in O$ which satisfies the following conditions:
\begin{align*}
b \equiv &\;x \pmod{P^2} \nonumber\\
b\equiv &\;1\pmod{P_i},\quad \forall \ i=1,\ldots,k.
\end{align*}
In particular, $bO$ is comaximal with $\mathfrak{f}$, hence the primary components of $bO$ are comaximal with the conductor, so, in particular, they are regular. Hence, by Lemma \ref{Neu}, we get a primary decomposition of $bO$ of the form
\begin{equation*}
bO=P \prod_{j=1}^n {P'_j}^{e_j}
\end{equation*}
where the $P'_j$'s are prime (i.e., maximal) ideals of $O$ distinct from the $P_i$'s.

We are in the position to apply Corollary \ref{generalcase}, hence there exists $a \in P$ such that $a, b$ form an idempotent pair and $P$ is the radical of $\langle a , b \rangle$. Since $\langle a , b \rangle$ is a product of primary ideals, $b \notin P^2$ and $P$ is regular, by Lemma \ref{Neu} it follows that $P = \langle a, b \rangle$, as required.

Finally, note that $O$ is integrally closed if and only if $\mathfrak{f}=O$. In this case, every prime ideal of $O$ is comaximal with the conductor, so every prime ideal can be generated by an idempotent pair. \qed
\end{proof}

\begin{Rem} 
The congruences in the proof of Proposition \ref{primeidealsorderPrinc} are not necessary. Indeed, the crucial conditions are:

(i) the $P$-primary component of $bO$ is regular.

(ii) $b\in P\setminus P^2$. 

In fact, if
$$
bO= P \prod_{j=1}^n Q_j
$$
where $Q_j$ are primary ideals of $O$ (not necessarily regular), we may take $a \in O$ satisfying the conditions:
\begin{align*}
a\equiv &\; b\pmod{P}\\
a\equiv &\;b+1\pmod{Q_j},\;\forall j=1,\ldots,k
\end{align*}
and the same conclusion follows: $a(1-a) \in b O$, and $P$ is the only primary ideal that contains both $a$ and $b$, so that $\langle a, b \rangle = P$, since $b \notin P^2$ and the $P$-primary ideals are regular. 

On the other hand, if $P$ were not regular we may not get the same conclusion, even imposing the condition $a\equiv b\pmod{P}$, since in this case the $P$-primary ideals are not linearly ordered. We still get that $\langle a, b \rangle$ is $P$-primary, though. We will see in Section 4 that the conductor of the order $\Z[\sqrt{-3}]$ is an instance of this phenomenon.

\end{Rem}

\begin{Corollary}\label{OPrinc}
If an order $O$ is a PRINC domain, then every regular ideal $I$ of $O$ is principal. 
\end{Corollary}

\begin{proof}
Let $I \subset O$ be a regular ideal. In particular, the primary components of $I$ are regular primary ideals. By Lemma \ref{Neu}, each of these primary components is equal to a power of its own radical. By Proposition \ref{primeidealsorderPrinc}, each of these radicals is generated by a suitable idempotent pair, hence they are all principal, since $O$ is a PRINC domain. Hence, all the primary components of $I$ are principal as well. It follows that $I$ is equal to a product of principal ideals, so it is principal.  \qed
\end{proof}

The next corollary is a consequence of Proposition \ref{primeidealsorderPrinc}; it also follows from Corollary \ref{PruferfinitePrinc}.

\begin{Corollary}\label{DedekindPrinc}
A Dedekind domain is PRINC if and only if it is a PID.
\end{Corollary}

\begin{Corollary}\label{OPrincDPrinc}
If an order  $O$  is a PRINC domain then its integral closure $D$ is a PRINC domain, or, equivalently, a PID. 
\end{Corollary}
\begin{proof}
We will show that each prime ideal $P$ of $D$ is principal.  Without loss of generality, we may just consider the case of a regular prime ideal $P$. In fact, there are only finitely many prime ideals of $D$ that divide the conductor, and, by Claborn \cite[Corollary 1.6]{Cla}, a Dedekind domain which is not a PID has an infinite number of non-principal prime ideals. Since $P\cap O$ is regular, too (see the remark below), it follows by Corollary \ref{OPrinc}  that $P\cap O$ is principal. Since $P$ is an extended ideal, $P=(P\cap O)D$, so $P$ is principal, too. \qed
\end{proof}

\begin{Rem}\label{idealscoprimeconductor}
By the arguments of the proof of Corollary \ref{OPrinc}, the following conditions are equivalent:
\begin{itemize}
\item[1)] each regular ideal of $O$ is principal.
\item[2)] each regular prime ideal of $O$ is principal.
\end{itemize}
If one of the two conditions holds, then $D$ is a PID, exactly by the same argument of the proof of Corollary \ref{OPrincDPrinc}: each regular prime ideal $P$ of $D$ is extended, so it is equal to the extension of its contraction, which is principal. More generally,  there is a 1-1 correspondence between the regular ideals of $D$ and the regular ideals of $O$ (see for example \cite[Lemma 2.26, p. 389]{PZ}). In particular, each regular ideal of $O$ is contracted and each regular ideal of $D$ is extended.
\end{Rem}

\section{$\Z [\sqrt {{- d}}]$ not satisfying (princ).}

Let $\eta = \sqrt {{- d}} $, where $d > 0$ is a square-free integer. 
In this section, we want to establish when the order $O = \Z [\eta]$ fails to be a PRINC domain. By Corollary \ref{OPrincDPrinc}, we know that $O$ cannot be PRINC if its integral closure $D$ is not a PID. So, it is enough to examine the cases when $D$ is a PID, namely when $d = 1, 2, 3, 7, 11, 19, 43, 67, 163$ (see, for instance, \cite{Ri}, p. 81). Of course, when $d = 1, 2$ we get $O = D$, hence $O$ is trivially PRINC. We will focus on the remaining cases. 

The next proposition is valid for any $d > 0$.

\begin{Proposition}\label{nonprincipalideal}
Let $O = \Z[\eta]$, with $\eta = \sqrt{-d}$, where $d > 0$ is any integer (not necessarily square-free). Then we have:

(1) The element $1 + \eta$ is irreducible in $O$.

(2) If $a \in \Z\setminus\{\pm1\}$ properly divides $1 + d$, then $\langle 1 + \eta, a \rangle$ is a proper non-principal ideal of $O$.  
\end{Proposition}

\begin{proof}
(1) Assume that $1 + \eta = (x + \eta y) (z + \eta t)$ for suitable $x, y, z, t \in \Z$. Taking norms, we get
$$
1 + d = (x^2 + d y^2)(z^2 + d t^2),
$$
which implies that $y=0$ or $t=0$. Assuming that $y=0$, it follows that $x\in\Z$ divides $1+\eta\in\Z[\sqrt{-d}]$, which implies that $x=\pm1$.

\medskip

(2) Firstly, let us show that $I=\langle 1 + \eta, a \rangle$ is a proper ideal of $O$. Assuming that $I$ is not proper, we obtain that
$$O=\langle 1 + \eta, a \rangle\langle 1 - \eta, a \rangle\subseteq aO$$
which is a contradiction.

Let us see that $I$ is not principal. Otherwise, we should get $I = (1 + \eta)O$, since $1 + \eta$ is irreducible, hence, in particular, $a = (1 + \eta)(x + \eta y)$, where $x, y \in \Z$, $y\not=0$. But this is impossible, since 
$$a^2 = (1 + d) (x^2 + d y^2) \ge (1+ d)^2 > a^2.$$
\qed
\end{proof}

It follows from Proposition \ref{nonprincipalideal} that, if $1 + d = a (a-1)$ for some $a \in \Z$, $a\not=-1$, then $O$ does not satisfy (princ). For example, $\Z[\sqrt{-11}]$ and $\Z[\sqrt{-19}]$ do not satisfy (princ). However, we will prove below a stronger result (Proposition \ref{non-princ}). 

Like Proposition \ref{nonprincipalideal}, also the next lemma is valid for any $d > 0$, not necessarily square-free.

\begin{Lemma}\label{prime ideal}
In the above notation, if $p \in \Z$ is a prime which divides $1 + d$, then the ideal $\langle p,1+\eta \rangle$ of $ \Z[\eta]$ is prime.
\end{Lemma}

\begin{proof}
Since $1+d=pb$, for some $b\in\Z$, then $X^2+d=(X+1)(X-1)+pb\in \langle p,1+X \rangle \Z[X]$. Hence, if $\pi:\Z[X]\to \Z[\eta]$ is the canonical homomorphism sending $X$ to $\eta$, then $\pi^{-1}(\langle p,1+\eta \rangle \Z[\eta])= \langle p,1+X \rangle \Z[X]$, since the latter ideal contains $\langle X^2+d \rangle $, the kernel of $\pi$. Therefore, 
$$
\frac{\Z[\eta]}{\langle p,1+\eta \rangle}\cong \frac{\Z[X]}{ \langle p,1+X \rangle} \cong \Z/p\Z.
$$
\qed
\end{proof}

\begin{Proposition} \label{non-princ}
Let $d \in \{ 11, 19, 43, 67, 163 \}$. Then $O = \Z[\sqrt{-d}]$ does not satisfy (princ).
\end{Proposition}

\begin{proof}
Note that for each of the relevant $d$ there exists a prime $p \ne 2$ that properly divides $1+d$. Lemma \ref{prime ideal} shows that $P=\langle p,1+\eta \rangle$ is a prime ideal of $O$, which is comaximal with the conductor $\mathfrak{f}$, since $2 \in \F$ and $p$ is odd. Then $P$ is generated by an idempotent pair, by Proposition \ref{primeidealsorderPrinc}. Moreover, since $p$ properly divides $1+d$ (which is the norm of $1+\eta$), by Proposition \ref{nonprincipalideal} the same ideal $\langle p,1+\eta \rangle$ is not principal. We conclude that $O$ does not satisfy (princ). \qed
\end{proof}

The argument in the proof of Proposition \ref{non-princ} neither applies to $\Z[\sqrt{-3}]$ nor to $\Z[\sqrt{-7}]$, since $1+3=4$ and $1+7=8$ are not divisible by an odd prime. In fact, we note that the above proof shows that, in the orders $\Z[\sqrt{-d}]$ with $d \in \{ 11, 19, 43, 67, 163 \}$, there are regular prime ideals that are not principal, but become principal after extending them to their integral closures, which are PIDs. This means that the generator of such an extended ideal lies in the integral closure but not in the corresponding (proper) order. This phenomenon does not happen with the orders $\Z[\sqrt{-3}]$ and $\Z[\sqrt{-7}]$, as we will see in the next section.

\section{$\Z[\sqrt{-3}]$ and $\Z[\sqrt{-7}]$ are PRINC domains.}\label{Princ}

We start recalling some well-known facts on $\Z[\sqrt{-d}] = \Z[\eta]$.
Let $d \in\Z$ be a positive square-free integer, which is congruent to $3$ modulo $4$. Then the ring of integers of $\Q(\eta)$ is $D=\Z[\frac{1+\eta}{2}]$. The conductor $\mathfrak f$ of $\Z[\frac{1+\eta}{2}]$ into the order $O=\Z[\eta]$ is equal to:
$$
\mathfrak{f}= 2D = \langle 2, 1+\eta \rangle O \subset O.
$$
Clearly, the conductor is a maximal ideal of $O$, since the quotient ring $O/\mathfrak{f}$ is isomorphic to $\Z/2\Z$. As an ideal of $O$, $\mathfrak{f}$ is not principal. In fact, Proposition \ref{nonprincipalideal} applies to $\mathfrak f = \langle 2, 1+\eta \rangle$, since $2$ properly divides $1+d$. More generally, it is straghtforward to show that the conductor of a proper order is always not principal. Moreover, a simple computation shows that $\mathfrak{f}^2=2\mathfrak{f}$; it follows that $\mathfrak{f}$ is not  invertible, since $2$ does not generate $\mathfrak{f}$.

The next technical lemma will be a main ingredient for the proof of the following Theorem \ref{primeidealsZsqrt-3-7}. We denote by $D^*$ the multiplicative group of the units of a domain $D$.

\begin{Lemma}\label{sqrt-3}
Let $D=\Z[\frac{1+\sqrt{-3}}{2}]$ and $O=\Z[\sqrt{-3}]$. Then for each $z\in D$, there exists a unit $u\in D^*$ such that $zu\in O$.
\end{Lemma}

\begin{proof}
We recall that $D^*=\{\pm1,\frac{\pm1\pm\sqrt{-3}}{2}\}$, namely, the multiplicative group of the $6$th roots of unity. If $z\in O$ the claim follows immediately. Let $z\in D\setminus O$; then we may write $z=\frac{a+b\sqrt{-3}}{2}$, for suitable integers $a, b$, with $a\equiv b\equiv 1\pmod 2$. We have
$$
\frac{a+b\sqrt{-3}}{2}\cdot \frac{1 + \sqrt{-3}}{2} = \frac{( a - 3b)+( a + b)\sqrt{-3}}{4}
$$
and
$$
\frac{a+b\sqrt{-3}}{2}\cdot \frac{1 - \sqrt{-3}}{2} = \frac{( a + 3b) - ( a - b)\sqrt{-3}}{4}
$$
Looking at the residue classes modulo $4$, a direct check shows that either $ a - 3b \equiv  a + b \equiv 0 \pmod 4$ or $ a + 3b \equiv  a - b \equiv 0 \pmod 4$, for any possible choice of the odd integers $a,b$. We conclude that $z u \in O$ for some $u \in D^*$. \qed
\end{proof}

We remark that $\Q(\sqrt{-1})$ and $\Q(\sqrt{-3})$ are the only imaginary quadratic number fields which contains roots of unity distinct from $\pm1$.

\begin{Theorem}\label{primeidealsZsqrt-3-7}
Let $\eta = \sqrt{-d}$, where $d\in\{3,7\}$. Let $P \subset\Z[\eta] = O$ be a prime ideal containing an odd prime $p$. Then $P$ is principal.
\end{Theorem}

\begin{proof}
Let $D=\Z[\frac{1+\eta}{2}]$ be the integral closure of $O$. Let $P$ be any prime ideal of $O$ containing an odd prime $p$. In particular, $P$ is regular and so it is a contracted ideal, namely, $PD\cap O=P$ (see Remark \ref{idealscoprimeconductor}).  We firstly examine the case when $p=d$. Then $p$ is ramified in $D$, and the unique prime ideal of $D$ above $p$ is $\sqrt{-p}D$. It follows that $P$ is principal, equal to $\sqrt{-p} \cdot O$. Suppose now that $p$ is an odd prime different from $d$. If $pD$ is a prime ideal, then $P$ is equal to $pD\cap O =pO$, and so $P$ is principal. Suppose that $p$ decomposes in $D$, so that $P$ is one of the two distinct prime ideals above $p$ in $O$. We know that $PD$ has norm $p$ (and so $P$ as well, since $p \ne 2$). Since $D$ is an Euclidean domain, it follows that $PD$ is principal, generated by an element $\alpha\in D$ of norm $p$. 

We distinguish now the two cases.
\begin{itemize}
\item[i)] $d=3$
\end{itemize}
By Lemma \ref{sqrt-3}, we may multiply $\alpha$ by a suitable unit of $\Z[\frac{1+\sqrt{-3}}{2}]$, to get an element $\alpha'\in \Z[\sqrt{-3}]$ which is associated to $\alpha$. In particular, $\alpha'$ is a generator of $PD$ which lies in $\Z[\sqrt{-3}]$, so that $\alpha'O = \alpha' D \cap O = PD \cap O = P$.

\medskip

\begin{itemize}
\item[ii)] $d=7$
\end{itemize}
Since $\alpha$ is an element of $\Z[\frac{1+\sqrt{-7}}{2}]$, we may write $\alpha=\frac{a+b\sqrt{-7}}{2}$, for suitable integers $a, b$, with $a\equiv b\pmod 2$. We have $N(\alpha)=\frac{a^2+7b^2}{4}=p$, or, equivalently,
\begin{equation}\label{equation}
a^2+7b^2=4p
\end{equation}
If $a\equiv b\equiv1\pmod 2$, then $a^2\equiv b^2\equiv 1\pmod 8$, so, looking at (\ref{equation}), we get $0 \equiv 4p \pmod 8$, a contradiction. This means that, necessarily,  $a\equiv b\equiv0\pmod 2$, i.e., the generator $\alpha$ of $PD$ actually in $\Z[\sqrt{-7}]$, hence $P=\alpha\Z[\sqrt{-7}]$ follows.
\qed
\end{proof}

\begin{Corollary}\label{cor1}
Let $d \in \{ 3, 7 \}$. Let $I\subset\Z[\sqrt{-d}]$ be a regular ideal. Then $I$ is principal. 
\end{Corollary}

\begin{proof}
The proof follows by Remark \ref{idealscoprimeconductor} and by Theorem \ref{primeidealsZsqrt-3-7}. \qed
\end{proof}

\begin{Corollary}\label{Z-3Z-7Princ}
The orders $\Z[\sqrt {- 3}]$ and $\Z[\sqrt {- 7}]$ are PRINC domains. 
\end{Corollary}

\begin{proof}
Let $O = \Z[\sqrt{-d}]$, where $d\in\{3,7\}$. Take any idempotent pair $a, b \in O$; assume, without loss of generality, that $a(1-a) = b c$, for some $c \in O$. Let us show that $\langle a, b \rangle$ is a principal ideal of $O$.
Assume, for a contradiction, that $\langle a, b \rangle$ is not principal. Then Corollary \ref{cor1} shows that $\langle a, b \rangle$ is contained in the conductor $\mathfrak{f}$, because $\mathfrak{f}$ is a maximal ideal. Then, obviously, the ideal $\langle 1 - a, b \rangle$ is comaximal with $\mathfrak{f}$, hence it is principal in $O$. However, recall that
$\langle a, b \rangle \langle 1- a, b \rangle  = b O$ (cf. the proof of Theorem \ref{idempotents}).
Since $\langle 1- a, b \rangle$ is principal, it follows that also $\langle a, b \rangle$ is principal, impossible. \qed
\end{proof}

We recall that U. Zannier privately communicated, to the third author, a first direct proof of the fact that $\Z[\sqrt{-3}]$ is a PRINC domain.

Since by Corollary the rings \ref{Z-3Z-7Princ} $\Z[\sqrt{-3}]$ and $\Z[\sqrt{-7}]$ are PRINC domains of finite character, by Corollary \ref{PruferfinitePrinc}  each projective ideal (hence, invertible) of these domains is principal. In the next proposition we give an {\it ad hoc} argument of this result.

\begin{Proposition}\label{invertibleideals}
Let $O$ be either $\Z[\sqrt{-3}]$ or $\Z[\sqrt{-7}]$. Then every invertible ideal of $O$ is principal.
\end{Proposition}
\begin{proof}
Take an arbitrary invertible ideal $I$ of $O$. Corollary \ref{cor1} shows that every ideal of $O$ not contained in the conductor $\mathfrak f$ is principal. By the unique factorization of ideals of $O$ into primary ideals, it follows that any element $s \in O \setminus \mathfrak f$ is a product of prime elements of $O$, and any ideal $I$ contained in $\mathfrak f$ has the form $I = s Q$, where $s \notin \mathfrak f$ and $Q$ is $\mathfrak f$-primary. Therefore, to prove our statement, it suffices to consider the case when $I = Q$ is $\mathfrak f$-primary. 

We must show that the invertible ideal $Q$ is principal. Let us consider the localization $O_{\mathfrak f}$ of $O$ at $\mathfrak f$. Then the extended ideal $Q_{\mathfrak f}$ is invertible in $O_{\mathfrak f}$, hence it is principal (since local domains are projective-free), say $Q_{\mathfrak f} = a O_{\mathfrak f}$, where we may take $a \in Q$.

By the unique factorization of ideals of $O$ into primary ideals, we readily get $aO = s Q_1$, where $Q_1$ is $\mathfrak f$-primary, and $s \in O \setminus \mathfrak f$, so it is a product of prime elements of $O$. Say $a = s a_1$; then $a_1 O$ is $\mathfrak f$-primary, equal to $Q_1$,  and we have $Q_{\mathfrak f} = a_1 O_{\mathfrak f}$. 

Pick now any element $b \in Q$. Then $b = a_1 y/t$, for some $t \in O \setminus \mathfrak f$. Say $t = \prod_1^n p_i$, where the $p_i$ are prime elements of $O$. Then from $t b = a_1 y$ and $a_1 \notin p_i O$ we get $y/p_i \in O$, for $i = 1, \dots, n$. It readily follows that $y/t \in O$, whence $b \in a_1 O$. Since $b \in Q$ was arbitrary, we conclude that $Q = a_1 O$, as required.  \qed
\end{proof}

We summarize the previous results in a final statement.
\medskip
\begin{Theorem} \label{final}
Let $d > 0$ be a square-free integer. Then $\Z[\sqrt{-d}]$ does not satisfy property (princ), except $\Z[\sqrt{-1}]$, $\Z[\sqrt{-2}]$, which are PIDs, and $\Z[\sqrt{-3}]$, $\Z[\sqrt{-7}]$.
\end{Theorem}

We remark that the next proposition can be proved using, instead of the UCS property, Corollary 2.6 in \cite{H}, quoted by Heitmann as "Serre's Theorem".

\begin{Proposition} 
The two rings $\Z[\sqrt{-3}]$ and $\Z[\sqrt{-7}]$ are projective-free.
\end{Proposition}

\begin{proof}
Let $O$ be one of the two considered rings. $O$ is almost local-global, since its proper quotients are zero-dimensional, hence local-global. Then, by a result by Brewer and Klingler \cite{BK} (see also Theorem  4.7 in Chapter V of  \cite{FS}), $O$ satisfies the UCS-property, that is, every finitely generated submodule $M$ of a free module $F$ with unit content contains a rank-one projective direct summand of $F$, and hence of itself. By Proposition \ref{invertibleideals}, this direct summand must be cyclic. Let now assume that $M$ is finitely generated projective, and let $I$ be its content. Let $F=M \oplus N$ be a free module containing $M$ as a summand. Then $M \subseteq IF=IM \oplus IN$ implies that $M=IM$, so that, by Nakayama's lemma, there exists an element of $O$ of the form $1-a$, with $a \in I$, such that $(1-a)M= 0$. But $M$ is torsion-free, hence $a=1$ and $I= O$. By the UCS property, $M$ contains a cyclic summand $x O$, that is, $M = M \oplus xO$. Now an easy induction on the rank of the projective module shows that $M$ is free. \qed
\end{proof}

\medskip
The following question naturally arises: are there PRINC domains which are neither UFD, nor projective-free?  An example of a domain of this kind was exhibited in \cite[Section 4]{MS1}. The authors refer also to a paper by Gilmer \cite{GilmNote}, where an example of a domain containing an $n$-generated invertible ideal ($n$ an arbitrary positive integer) was given. In \cite[Remark p. 189]{MS1}, the authors show that the domain in the example of Gilmer is a PRINC domain which is neither UFD nor projective-free.\\

The following question is still open.\\

\noindent{\bf Question}: Let $R$ be a Pr\"ufer domain which is PRINC. Is $R$ a B\'ezout domain?\\

We remark that a positive answer was given in \cite[Corollary 1.9]{MS1} under the additional hypothesis that the Pr\"ufer domain is CFD.
\bigskip

\noindent{\bf Acknowledgements} Research supported by  ``Progetti di Eccellenza 2011/12" of Fondazione CARIPARO and by the grant "Assegni Senior" of the University of Padova.
We are grateful to the referee for suggesting many improvements to the paper, especially the useful characterization in Theorem \ref{idempotents}.

\end{document}